\pdfoutput=1

\documentclass[11pt, a4paper]{article}

\usepackage{sty/dl-en}
\usepackage{sty/config}

\DeclareMathOperator{\Ran}{{\mathrm{Exp}^\ast}}
\DeclareMathOperator{\TopRan}{\mathrm{Exp}^\ast_{\mathrm T}}

\DeclareMathOperator{\MetricRan}{\mathrm{Exp}^\ast_{\mathrm M}}

\DeclareMathOperator{\Exp}{{\mathrm{Exp}}}

\DeclareMathOperator{\TopExp}{{\mathrm{Exp}}_{\mathrm T}}

\DeclareMathOperator{\MetSphere}{{\mathrm{S}}_{\mathrm M}}

\DeclareMathOperator{\MetricExp}{{\mathrm{Exp}}_{\mathrm M}}

\DeclareMathOperator{\ExpMin}{{\mathsf{Exp}}_\vee}
\DeclareMathOperator{\Exit}{{\mathsf{Exit}}}
\DeclareMathOperator{\Cone}{{\mathrm{C}}}
\DeclareMathOperator{\Conv}{{\mathrm{Conv}}}

\newcommand{\FiniteProduct}[1]{\prod_{#1}{}^{\raisebox{4pt}{\( \mathrm{f} \)}}}
\newcommand{\SmallFiniteProduct}[1]{\prod_{#1}^{\mathrm f}}
\newcommand{\TinyFiniteProduct}{\times^{\mathrm f}}

\newtheorem{construction}[theorem]{Construction}
\newtheorem{Open question}[theorem]{Open question}

\newcommand{\Vect}{\mathsf{Vect}}
\newcommand{\PointedOmega}{\OrdinalOmega_\ast}
\newcommand{\TruncatedTopExp}[1]{
{\mathrm{Exp}}^{\leq_\ast #1}_{\mathrm T}}
\newcommand{\TruncatedMetExp}[1]{
{\mathrm{Exp}}^{\leq_\ast #1}_{\mathrm M}}
\newcommand{\TruncatedExp}[1]{{\mathrm{Exp}}^{\leq_\ast #1}}
\newcommand{\StratumExp}[1]{\mathrm{Exp}_\mathrm{M}^{#1}}

\newcommand{\Circle}{\mathbf{S}^1}
\newcommand{\Torus}[1]{\mathbf{T}^{#1}}
\newcommand{\OrdinalOmega}{\omegaup}
\newcommand{\Sets}{\mathsf{Set}}
\newcommand{\Op}{^\mathrm{op}}
\newcommand{\SymmetricGroup}[1]{\mathbf{S}_{#1}}
\renewcommand{\Top}{\mathsf{Top}}
\newcommand{\PiZero}{\piup_0}
\newcommand{\PiN}[1]{\piup_{#1}}
\newcommand{\GenMet}{\mathsf{Met}_\infty}

\newcommand{\TOne}{\mathsf{T}_1}
\newcommand{\IdemTwo}{\mathbf{I}_2}
\newcommand{\Radius}{\mathrm{r}}
\newcommand{\Center}{\mathrm{c}}
\newcommand{\lInfinity}{\ell^∞}
\newcommand{\Ball}{\mathrm{B}}

\newcommand{\Category}[1]{\mathcal #1}
\newcommand{\CommutativeMonoids}{\mathsf{Commutative~monoids}}
\newcommand{\FinSurj}{\mathsf{Fin}_{\twoheadrightarrow}}
\newcommand{\Fin}{\mathsf{Fin}}
\newcommand{\Slice}[2]{{#1}_{/#2}}
\newcommand{\CoSlice}[2]{{#1}_{#2/}}
\newcommand{\Fact}[1]{\mathcal #1}

\newcommand{\UnskipRef}[1]{\unskip~\textnormal{[\ref{#1}]}}

\newcommand{\SectionRef}[1]{\unskip~[\hyperref[#1]{\S\thinspace\ref{#1}}]}

\theoremstyle{remark}
	\newtheorem{warning}[theorem]{Warning}

\begin{document}

\maketitle

\footnotefirstpage

\begin{abstract}
	Factorization algebras have been defined using three different
	topologies on the Ran space.
	We study these three different topologies on the exponential, which
	is the union of the Ran space and the empty configuration, and
	show that an exponential property is satisfied in each case.
	As a consequence we describe the weak homotopy type 
	of the exponential \( \Exp(X) \) for each topology, in the case where
	\( X \) is not connected.

	We also study these exponentials as stratified spaces and
	show that the metric exponential is conically stratified
	for a general class of spaces.
	As a corollary, we obtain that locally constant factorization
	algebras defined by Beilinson-Drinfeld are equivalent to
	locally constant factorization algebras defined by Lurie.
\end{abstract}

\tableofcontents

\newpage

Roughly speaking,
a factorization algebra \( \Fact A \)
on a space \( X \) with
values in a symmetric monoidal category \( \Category C^\otimes \)
is a gadget associating
to each finite subset of points \( S \subset X \) an object
\( \Fact A_S \in \Category C \), such that
\begin{description}
	\item[factorization]
		\( \Fact A_{\sqcup_{i \in I} S_i}
		\IsIsomorphicTo \otimes_{i \in I} \Fact A_{S_i} \)
		for every finite family \( I \) of disjoint
		finite subsets \( S_i \subset X \);
	\item[continuity]
		the assignment \( S \mapsto \Fact A_S \) be continuous.
\end{description}
Yet,
to be able to say that
\( \Fact A_S \) varies continuously with \( S \), one first
needs to answer the question:

\begin{center}
	\emph{What is the topology
	on the set of all finite subsets \( S \subset X \)?}
\end{center}

The set of all finite subsets of \( X \) is called the exponential
\( \Exp(X) \) of \( X \).
The literature on factorization algebras provides three different
candidates to topologize \( \Exp(X) \):
\begin{description}
	\item[2004] In \emph{Chiral Algebras}
		\cite[3.4.1]{doi:10.1090/coll/051}, Beilinson and Drinfeld
		endow \( \Exp(X) \) with a colimit
		topology;
	\item[2009] In \emph{Derived Algebraic Geometry VI}
		\cite[3.3.2]{arXiv:0911.0018}, Lurie
		endows \( \Exp(X) \) with a topology reminiscent
		of the metric topology introduced by Hausdorff on the
		space of compact subsets of a metric space;
	\item[2016] In \emph{Factorization Algebras in Quantum Field Theory}
		\cite[1.4.1]{doi:10.1017/9781316678626},
		Costello and Gwilliam use yet another topology to define
		factorization algebras, this time using coverings inspired
		from Weiss.
\end{description}

For a given separated \( X \), the three topologies above are given
from finest to coarsest and one then obtains three different
levels of strength for the continuity requirement of a factorization
algebra.
It has been conjectured that the three different definitions
agree in the special case of \emph{locally constant} factorization
algebras which are roughly speaking, those factorization algebras
\( \Fact A \) for which \( \Fact A_x \) is ``homotopic'' to
\( \Fact A_y \) when \( x, y \in X \) both belong to the same
contractible open subset.

The set of all finite subsets of \( X \) is called the exponential
of \( X \) because its algebraic properties resemble   
that of exponential functions.
This is the subject of the first section, where we define
\emph{exponential functors} in general and give some general properties.

In the second section, we introduce the three different topologies
giving rise to the topological exponential (B\&D version),
the metric exponential (Lurie version) and the minimal exponential
(C\&G version).
We show how each exponential listed above is an exponential functor
in the sense of the definition given in the first section.
From this we deduce the weak homotopy type of each exponential
in the case where \( X \) is not connected,
extending contractibility results
of Handel
\cite[4.3]{hjm:28-4}
and Curtis \& Nhu
\cite{doi:10.1016/0166-8641(85)90005-7}.

Finally, the last two sections are dedicated to the study of the
stratification of the metric exponential.
The goal is to show
that it is conically stratified (in the sense of Lurie) for
a general class of spaces.
For this we need to solve an optimization problem: finding
the smallest enclosing ball of a finite number of points in
a general normed space; this is the content of the third section.
Using a companion article from the second author
\cite{arXiv:2102.12325}, one can then deduce
from the conical stratification of the metric exponential
that the notions of
locally constant factorization algebras from
Beilinson \& Drinfeld and Lurie
coincide.

\section{Exponentials}

\subsection{Exponential functors}

Any continuous function \( f \From \Reals \to \Reals \)
satisfying \( f(x+y) = f(x) f(y) \) must be an exponential
\( x \mapsto a^x \) with base \( a = f(1) \).
\emph{The} exponential is traditionally the one with base
\( \e \) defined
\[
	\e^x \coloneqq \sum_{n \in \Naturals} \frac{x^n}{n!}
\]
using power series.

An analogous theory can be described in the realm of categories.
The set \( \Reals \) can be replaced by
any category \( \Category C \), functions can be replaced by functors,
sums can be replaced by coproducts
and products can be replaced by categorical products.

However, there shall be two main differences
between exponential functors and exponential functions.
First, what was a \emph{property} of a function in
the realm of set theory
shall become a \emph{structure} on a functor in category theory.
An exponential functor
shall be a \emph{symmetric monoidal functor}
\[
	\begin{tikzcd}
		\textstyle
		\Category C^\sqcup
		\rar["E"]
			& \Category C^\times
	\end{tikzcd}
\]
between \( \Category C \) endowed with the
coproduct symmetric monoidal structure and
\( \Category C \) endowed with the product symmetric monoidal
structure.

Let us see some of the first obvious consequences.
First, since each object \( X \) admits a map
\( \emptyset_{\Category C} \to X \) with source
the initial object of \( \Category C \), one gets a map
\[
	\ast_{\Category C}
	\IsIsomorphicTo E(\emptyset_{\Category C})
	\longrightarrow E(X)
\]
with source the
terminal object of \( \Category C \),
so each \( E(X) \) is a \emph{pointed object}
of \( \Category C \).
Since every object \( X \in \Category C \) is a commutative monoid
with product map \( X \amalg X \to X \) the fold map, it follows
that \( E(X) \) is also a commutative monoid with
composition
\[
	E(X) \times E(X) \IsIsomorphicTo E(X \amalg X)
	\longrightarrow E(X)
\]
and with unit the pointing already described.

Second, one needs to replace continuity with an equivalent notion.
There is already a notion of continuity for functors in category theory:
a functor is (co)continuous if it preserves small (co)limits.
This is unreasonable to ask.
Instead, let us rewrite an equivalent definition for the continuity
of an exponential function: a function \( f \From \Reals \to \Reals \)
for which \( f(x+y)=f(x)f(y) \) for every \( x, y \in \Reals \)
is continuous if and only if for every converging series
\( \sum x_n = x\), the sequence
with general term \( \prod_{i \leq n} f(x_i) \) converges to
\( f(x) \).
In category theory convergence is replaced by existence
and ``limit of a sequence'' can be replace with the 
``colimit of filtered diagram''.
To make this precise, we first recall the notion of the \emph{finite
product} of an infinite family.

\begin{definition}[(Finite product)]
	Let \( \Category C \) be a category with finite products
	and filtered colimits.
	Given a small family of pointed objects \( \{X_i\}_{i \in I} \)
	in \( \Category C \), let
	\[
		\FiniteProduct{i \in I} X_i
		\coloneqq
		\varinjlim_{
			\substack{J \subset I \\ J \text{ finite}}
		}
		\prod_{j \in J} X_j
	\]
	denote the finite product (or weak product, or restricted product)
	of the family obtained
	by taking the colimit over all finite subsets
	\( J \subset I \).
\end{definition}

\begin{example}
	In the category of vector spaces (seen as pointed via their zero
	vector), the finite product of a small family
	\( \{V_i\}_{i \in I} \),
	\[
		\FiniteProduct{i \in I} V_i \IsCanonicallyIsomorphicTo
		\bigoplus_{i \in I} V_i
	\]
	coincides with their direct sum.

	If \( \{X_i\}_{i \in I} \) is a small family of pointed topological
	spaces, then one has
	\[
		\begin{tikzcd}
			\displaystyle
			\FiniteProduct{i \in I} X_i
			\rar
				& {\displaystyle \prod_{i \in I} X_i}^{\text{box}}
		\end{tikzcd}
	\]
	a continuous injection from the finite product to the product,
	endowed with the box topology.
	When the points \( \ast \to X_i \) are all open,
	this map becomes an open embedding.
	In this case a basis of opens of the finite product is given
	by the images of the products \( \prod_{j \in J} U_j \) with
	\( J \subset I \) finite and \( U_j \subset X_j \) open.
\end{example}

\begin{construction}
	Let \( E \From \Category C^\sqcup \to \Category C^\times \)
	be a symmetric monoidal functor with \( \Category C \)
	having enough limits and colimits.
	Let \( \{X_i\}_{i \in I} \) be a small family of
	elements of \( \Category C \), then for each finite
	subset \( J \subset I \), one gets a map
	\[
		\textstyle
		\prod_{i \in J} E(X_i)
		\IsIsomorphicTo
		E\left(\coprod_{i \in J} X_i\right)
		\longrightarrow E\big(\coprod_{i \in I} X_i\big)
	\]
	using the functoriality of \( E \) and its monoidal
	structure.
	Moreover for any \( K \subset J \), the following
	diagram
	\[
		\begin{tikzcd}
			\prod_{i \in J} E(X_i)
			\rar["\IsIsomorphicTo"]
				&
				E\big(\coprod_{i \in J} X_i\big)
				\ar[rd]
				& \\
				&
				& E\big(\coprod_{i \in I} X_i\big) \\
			\prod_{i \in K} E(X_i)
			\ar[uu]
			\rar["\IsIsomorphicTo"]
				& E\big(\coprod_{i \in K} X_i\big)
				\ar[uu]
				\ar[ru]
		\end{tikzcd}
	\]
	commutes by functoriality of \( E \) and its monoidal
	structure.
	One then obtains a canonically defined map
	\[
		\FiniteProduct{i \in I} E(X_i)
		\longrightarrow
		E\left(\coprod_{i \in I} X_i\right)
	\]
	from the finite product of \( \{E(X_i)\}_{i \in I} \).
\end{construction}

\begin{definition}[(Exponential functor)]
	Let \( \Category C \) be a category with finite products,
	filtered colimits and small coproducts.
	An exponential functor of the category \( \Category C \)
	is a symmetric monoidal functor
	\[
		\begin{tikzcd}
			\textstyle
			\Category C^\sqcup
			\rar["E"]
				& \Category C^\times
		\end{tikzcd}
	\]
	between \( \Category C \) endowed with the
	coproduct symmetric monoidal structure and
	\( \Category C \) endowed with the product symmetric monoidal
	structure,
	for which in addition, the canonical map
	\[
		\textstyle
		\SmallFiniteProduct{i \in I} E(X_i)
		\longrightarrow
		E\big(\coprod_{i \in I} X_i\big)
	\]
	is an isomorphism,
	for every small family \( \{X_i\}_{i \in I} \) of objects of
	\( \Category C \).
	A morphism between two exponential functors
	is the data of a monoidal natural transformation.
\end{definition}

\begin{remark}
	One could define exponential functors in the following more
	abstract way. The coproduct is an \emph{infinitary} symmetric
	monoidal structure: the operations
	\( \{X_i\}_{i \in I} \mapsto \amalg_{i \in I} X_i \)
	are associative, symmetric and unital in an obvious way.
	Similarly, the finite product endows the category of pointed objects
	\( \Category C_\bullet \)
	with another infinitary symmetric monoidal structure.

	An exponential functor can then be defined as an infinitary
	symmetric monoidal functor
	\( \Category C^{\sqcup}
	\to \Category C^{\TinyFiniteProduct}_\bullet \).
	Here it happens that any (finitary) lax monoidal functor
	\( \Category C^{\sqcup} \to \Category C^{\times}_\bullet \)
	gives rise to an infinitary lax monoidal functor
	\( \Category C^{\sqcup}
	\to \Category C^{\TinyFiniteProduct}_\bullet \)
	which allows us to define exponential functors without
	first having to develop the theory of infinitary monoidal
	categories.
\end{remark}

The classification result for exponential functions on \( \Reals \)
has an equivalent form in the case of exponential functors
on the category of sets: these are classified by their base.

\begin{definition}[(Exponential of base \( A \))]
	Let \( A \) be a commutative monoid.
	The exponential of base \( A \) is the
	endofunctor of the category of sets
	defined by
	\[
		\Exp_A(X) \coloneqq
		\{ \phi \From X \to A
		\mid \phi^{-1}(0) \text{ is cofinite} \}
	\]
	for any set \( X \).
	If \( f \From X \to Y \) is a function and
	\( \phi \From X \to A \) is almost null,
	its image by \( \Exp_A(f) \) is the
	function \( \psi \From Y \to A \), where
	\[
		\psi(y) \coloneqq \sum_{x \in f^{-1}(y)} \phi(x)
	\]
	for any \( y \in Y \).
	The function \( \psi \) is well defined
	and almost null since \( A \) is commutative and
	\( \phi \) is almost null.
	The exponential structure of \( \Exp_A \) is
	straightforward.
\end{definition}

\begin{theorem}
	The assignment \( A \mapsto \Exp_A \) induces an equivalence
	\[
		\CommutativeMonoids
		\IsCanonicallyIsomorphicTo
		\mathsf{Set~exponentials}
	\]
	between the category of commutative monoids
	and the category of exponential functors of the category
	of sets.
\end{theorem}

\begin{proof}
	The assignment \( A \mapsto \Exp_A \) is functorial, its
	inverse takes an exponential \( E \) and extracts
	the commutative monoid \( E(\ast) \).
	By construction one has a canonical
	isomorphism \( \Exp_A(\ast) \IsCanonicallyIsomorphicTo A \)
	and the maps
	\( \Exp_A(\emptyset) \to \Exp_A(\ast) \)
	and
	\( \Exp_A(\ast) \times \Exp_A(\ast)
	= \Exp_A(\ast \amalg \ast) \to \Exp_A(\ast) \)
	recover the commutative monoid structure on \( A \).

	Conversely if \( E \) is an exponential, let \( A \) denote
	the commutative monoid \( E(\ast) \).
	Then one has
	\( E(X) \IsIsomorphicTo \SmallFiniteProduct{X} A = \Exp_A(X) \)
	for every set \( X \).
	Lastly, let us show that
	\( E(f) \IsIsomorphicTo \Exp_A(f) \)
	for every function \( f \From X \to Y \).
	The case where \( Y \) is a singleton and \( X \) is finite is
	true by construction and corresponds to
	the monoid structure \( \prod_X A \to A \) of \( A \).
	Taking the colimit over finite subsets gives us the
	case \( \SmallFiniteProduct X A \to A \)
	where \( X \) is infinite.
	Finally, the general case is obtained by
	writting a function \( f \From X \to Y \) as a disjoint
	union \( f_y \From X_y \to \{y\} \) with \( y \in Y \):
	\( E(f) \IsIsomorphicTo \SmallFiniteProduct{y \in Y}
	E(f_y) \IsIsomorphicTo \Exp_A(f_y) \IsIsomorphicTo
	\Exp_A(f) \).
	The natural isomorphism
	\( \Exp_A \IsIsomorphicTo E \)
	we have just described, is monoidal by construction.

	It is straightforward to check that
	\( \Exp_A(\ast) \IsCanonicallyIsomorphicTo A \)
	is natural in \( A \)
	and \( \Exp_{E(\ast)} \IsIsomorphicTo E \)
	is natural in \( E \).
\end{proof}

\begin{example}
	The exponential of base \( \Naturals \)
	\[
		\Exp_\Naturals(X)
		\IsCanonicallyIsomorphicTo
		X^0 \amalg X
		\amalg X^2_{\SymmetricGroup 2}
		\amalg X^3_{\SymmetricGroup 3}
		\amalg \cdots
	\]
	is the exponential functor corresponding
	to the analyst exponential
	of base \( \e \).

	Let \( \IdemTwo \) be the idempotent
	commutative monoid on two elements.
	Then the exponentials of base
	\( \IdemTwo \) and \( \Integers_2 \)
	have identical sets
	\[
		\Exp_{\IdemTwo}(X)
		\IsCanonicallyIsomorphicTo
		\Exp_{\Integers_2}(X)
		\IsCanonicallyIsomorphicTo
		\{ S \subset X
		\mid S \text{ is finite}\}
	\]
	but their monoid structures are different:
	for the exponential of base \( \IdemTwo \), the
	pair \( (\{x\}, \{x\}) \) is sent
	to \( \{x\} \) whereas for that of base
	\( \Integers_2 \), it is sent to
	\( \emptyset \).
\end{example}

\begin{remark}
	One can extend the definition of an
	exponential functor to accommodate any
	monoidal structure on the target.
	For example, the classification theorem above
	also holds for exponential functors
	\( \Vect_\Reals^\oplus
	\to \Vect_\Reals^\otimes \):
	they are equivalent to unital commutative
	\( \Reals \)-algebras.

	The exponential of base
	\( \Reals[\mathrm{X}] \)
	is the symmetric algebra functor.
	The exponential of base
	\( \Reals[\Integers_2] \) is the
	antisymmetric algebra functor.
\end{remark}

\subsection{\emph{The} exponential functor}

As is apparent from the definition of the exponential functors
with bases, there is a preferred exponential, the exponential of base
\( \IdemTwo \), which we shall refer to as \emph{the} exponential
functor, and denote it simply by \( \Exp \).

For any set \( X \), \( \Exp(X) \) can be identified with the set of
all finite subsets of \( X \).
For each function \( f \From X \to Y \), the associated function
\( \Exp(f) \From \Exp(X) \to \Exp(Y) \) sends a finite subset
\( S \subset X \) to \( f(S) \subset Y \).

The exponential can also be described as a particular colimit.
This is the definition one can use to define the exponential in
a general category.

\begin{definition}
	Let \( \Category C \) be a category
	admitting finite products and small colimits.
	Let \( \FinSurj \) denote the category of finite sets and
	surjections.
	Given an object \( X \in \Category C \) and a surjection
	\( \phi \From I \twoheadrightarrow J \)
	between two finite sets, one gets a split monomorphism
	\[
		X^\phi \From X^J \hookrightarrow X^I
	\]
	which means that \( X \) defines a functor
	\( \FinSurj\Op \to \Category C \).

	The exponential functor on \( \Category C \) is the colimit
	\[
		\textstyle
		\Exp(X) \coloneqq \varinjlim_{I \in \FinSurj\Op} X^I
	\]
	of the functor \( X \From \FinSurj\Op \to \Category C \),
	with the convention that \( X^\emptyset \) is the terminal
	object of \( X \), for every \( X \in \Category C \).
\end{definition}

We shall first make a remark about the structure of this colimit
and then show its universal property.

\begin{definition}
	Let \( \PointedOmega \) denote the poset
	\[
		\PointedOmega \coloneqq
	\{0\} \amalg \{1 < 2 < \dots < n < \cdots \}
	\]
	and let denote by
	\( \leq_\ast \) the associated partial order.
\end{definition}

The opposite category of the category of finite sets
and surjections admits a canonical functor to
\( \PointedOmega \) sending a finite set \( I \) to its
cardinal.
Hence the colimit defining \( \Exp \) can be computed in two
steps.

\begin{notation}
	Let
	\[
		\textstyle
		\TruncatedExp n (X) \coloneqq
		\varinjlim_{|I| \leq_\ast n}
			X^I
	\]
	for every natural \( n \).
	For example, if \( X \in \Sets \),
	\( \TruncatedExp n (X) \) is the
	set of all non-empty finite subsets \( S \subset X \) having
	at most \( n \) elements.

	Since \( \PointedOmega \) has an isolated point, we shall let
	\[
		\textstyle
		\Ran(X) \coloneqq \varinjlim_{0 < |I|} X^I
	\]
	be the subobject called the ``Ran space'' of \( X \) 
	in some parts of the literature.
\end{notation}

\begin{remark}
	By construction
	\[
		\textstyle
		\Exp(X) \IsCanonicallyIsomorphicTo
		\varinjlim_{n \in \PointedOmega}
		\TruncatedExp n (X)
		\qand
		\Ran(X) \IsCanonicallyIsomorphicTo
		\varinjlim_{n > 0} \TruncatedExp n (X)
	\]
\end{remark}

\begin{theorem}[(The exponential is an exponential)]
	Let \( \Category C \) be a category with finite products
	and small colimits.
	Assume moreover that \( Y \mapsto X \times Y \)
	commute with all small colimits for every \( X \in \Category C \).
	Then \emph{the} exponential
	\( \Exp \From \Category C \to \Category C \)
	has the canonical structure of an exponential functor.
\end{theorem}

\begin{proof}
	For each finite set \( I \), the functor \( X \mapsto X^I \)
	commutes with filtered colimits.
	Hence we have
	\[
		\textstyle
		\varinjlim_{
			\substack{J \subset K \\
				J~\mathrm{finite}
		}} \Exp\left(\coprod_{j \in J} X_j\right)
		\IsCanonicallyIsomorphicTo
		\Exp\big(\coprod_{k \in K} X_k\big).
	\]
	So it shall be enough to show that \( \Exp \) is
	a symmetric monoidal functor.

	Let \( \{X_k\}_{k \in K} \) be a finite family of elements
	of \( \Category C \).
	One has a sequence of canonical isomorphisms
	\begin{align*}
		\textstyle
		\Exp\big(\coprod_{k \in K} X_k\big)
			& = \textstyle \varinjlim_I \big(\coprod_{k \in K} X_k\big)^I
				& &
		\\
			& \IsCanonicallyIsomorphicTo
			\textstyle
			\varinjlim_I \coprod_{(I \to K)}
			\prod_{k \in K} X_k^{I_k}
				& & \text{(by distributivity)}
		\\
			& \IsCanonicallyIsomorphicTo
			\textstyle
			\varinjlim_{(I \to K)}
			\prod_{k \in K} X_k^{I_k}
				& & \text{(by cofinality, see 3)}
		\\
			& \IsCanonicallyIsomorphicTo
			\textstyle
			\varinjlim_{\{I_k\}_{k \in K}}
			\prod_{k \in K} X_k^{I_k}
				& & \text{(by isomorphy, see 2)}
		\\
			& \IsCanonicallyIsomorphicTo
			\textstyle
			\prod_{k \in K} \varinjlim_{I_k} X_k^{I_k}
				& & \text{(by distributivity)}
		\\
			& = \textstyle \prod_{k \in K} \Exp(X_k)
	\end{align*}
	where:
	\begin{enumerate}
		\item
			given a map of finite sets
			\( \phi \From I \to K \), we let
			\( I_k \coloneqq \phi^{-1}(k) \) for each \( k \in K \);
		\item
			the coproduct induces an isomorphism of categories
			\[
				\Fin^K
				=
				\Slice {\Fin} K
			\]
			sending \( \{I_k\}_{k \in K} \)
			to \( \coprod_{k \in K} I_k \to K \);
		\item
			given a small category \( \Category C \) and an object
			\( K \in \Category C \),
			let \( p \From \CoSlice {\Category C} K \to \Category C \)
			denote the forgetful functor.
			Then for every \( x \in \Category C \), the canonical map
			\[
				p^{-1}(x) \longrightarrow \Slice p x
			\]
			is cofinal.
			Moreover the fiber \( p^{-1}(x) \) is discrete.
			Thus for any functor \( F \) with source
			the coslice \( \CoSlice {\Category C} K \), its colimit
			can be computed as
			\[
				\textstyle
				\varinjlim_{(K \to x)} F
				\IsCanonicallyIsomorphicTo
				\varinjlim_x \varinjlim_{\Slice p x} F
				\IsCanonicallyIsomorphicTo
				\varinjlim_x \varinjlim_{p^{-1}(x)} F
				=
				\varinjlim_x \coprod_{(K \to x)} F.
			\]
	\end{enumerate}
	It is straightforward to check that these canonical isomorphisms
	endow \( \Exp \) with the structure of an exponential functor.
\end{proof}

\begin{remark}
	If \( Y \mapsto X \times Y \) only commutes with coproducts, then
	one can show that we still get structural maps
	\(
		\Exp(\amalg_{i \in I} X_i) \to
		\SmallFiniteProduct{i \in I} \Exp(X_i)
	\)
	turning the exponential into what one would call an oplax
	infinitary symmetric monoidal functor.
\end{remark}

\begin{corollary}
	Under the same assumptions, for every \( X \in \Category C \),
	\( \Exp(X) \) is the free idempotent commutative monoid on \( X \).
\end{corollary}

\begin{proof}
	Since \( \Exp \) is an exponential functor, \( \Exp(X) \) is
	a commutative monoid as explained earlier.

	It is idempotent for the following reason: for every finite set
	\( I \), the diagram
	\[
		\begin{tikzcd}
			X^I \rar["\text{diagonal}"]
			\dar
				& X^I \times X^I
				\rar
				\dar
					& (X \amalg X)^{I \sqcup I}
					\rar["\text{fold}"]
					\dar
						& X^{I \sqcup I}
						\dar
			\\
			\Exp(X)
			\rar
				& \Exp(X) \times \Exp(X)
				\rar["\IsIsomorphicTo"]
					& \Exp(X \amalg X)
					\rar
						& \Exp(X)
		\end{tikzcd}
	\]
	where all vertical maps are canonical maps,
	commutes.
	Moreover the top composite map equals the map
	\( X^I \to X^{I \sqcup I} \) induced by the fold map
	\( I \sqcup I \to I \). Hence the full top right composite map
	is again the canonical map \( X^I \to \varinjlim_J X^J \).
	Since this is true for every finite \( I \), this shows that
	the bottom composition is the identity of \( \Exp(X) \).

	Let \( (M, \mu) \) be a commutative
	and idempotent monoid in \( \Category C \).
	Assume a given map \( \psi \From X \to M \).
	Then for every finite \( I \), one has a well defined map
	\[
		\begin{tikzcd}
			X^I \rar["\psi^I"]
				& M^I
				\rar["\mu"]
					& M
		\end{tikzcd}
	\]
	because \( \mu \) is associative and commutative.
	Moreover, because \( \mu \) is idempotent, the diagram
	\[
		\begin{tikzcd}[column sep = huge]
			X^I \rar["\psi^I"] \dar[hook, "X^\phi" swap]
				& M^I
				\dar[hook, "M^\phi" swap]
				\ar[dr, "\mu"]
					&
					\\
				X^J \rar["\psi^J"]
				& M^J
				\rar["\mu"]
					& M
		\end{tikzcd}
	\]
	commutes for every \( \phi \From I \twoheadrightarrow J \).
	One thus gets a map \( \tilde{\psi} \From \Exp(X) \to M \)
	extending \( \psi \).
	The map \( \tilde{\psi} \) is obviously unital.

	To show that \( \tilde{\psi} \) is compatible with \( \mu \) is
	to show that
	\[
		\begin{tikzcd}[ampersand replacement=\&, column sep = huge]
			\Exp(X) \times \Exp(X)
			\arrow[r, "\tilde{\psi} \times \tilde{\psi}"]
			\arrow[d, "", swap]
			\& M \times M
			\arrow[d, "\mu"] \\
			\Exp(X)
			\arrow[r, "\tilde{\psi}"]
			\& M
		\end{tikzcd}
	\]
	commutes, which can be done by precomposing with
	\( X^I \times X^J \) for all \( I, J \) finite sets.
	We then only need to show that
	\[
		\begin{tikzcd}[column sep = huge, row sep = large]
			X^I \times X^J
			\dar
			\rar["\psi^I \times \psi^J"]
				& M^I \times M^J
				\dar
				\rar["\mu \times \mu"]
					& M \times M
					\dar[dd, "\mu"]
			\\
			(X \amalg X)^{I \sqcup J}
			\dar["\text{fold}" swap]
			\rar["(\psi \amalg \psi)^{I \sqcup J}"]
				& (M \amalg M)^{I \sqcup J}
				\dar["\text{fold}" swap]
			\\
			X^{I \sqcup J}
			\rar["\psi^{I \sqcup J}"]
				& M^{I \sqcup J}
				\rar["\mu"]
					& M
		\end{tikzcd}
	\]
	commutes.
	The left two squares commute by functorality.
	The right square commutes because \( \mu \) is unital, commutative
	and associative.
\end{proof}

\section{Topologies on the exponential}

In this section we shall review three different topologies on the
set \( \Exp(X) \) of finite subsets \( S \subset X \)
of a topological space \( X \),
with the goal of transforming \( \Exp \) into
an exponential endofunctor of the category of topological spaces.

\subsection{The topological exponential}

As explained in the previous section, the exponential can be computed
with the help of a colimit ranging over the opposite category
of the category of finite
sets and surjections.
Computing the colimit in the category of topological spaces, one
obtains the \emph{topological exponential}, of which we give a simpler
definition.

\begin{definition}[(Topological exponential)]
	The topological exponential
	of a topological space \( X \)
	is the topological space denoted \( \TopExp(X) \) with set of points
	\( \Exp(X) \), the set of finite subsets \( S \subset X \),
	endowed with the finest topology such that
	the canonical maps
	\[
		X^n \longrightarrow \Exp(X)
	\]
	given by sending each tuple
	\( (x_1, \dots, x_n) \) to the subset
	\( [x_1, \dots, x_n] \subset X \) it represents,
	be continuous for every \( n \geq 0 \).
\end{definition}

The topological exponential is used by Beilinson and Drinfeld
\cite[3.4.1]{doi:10.1090/coll/051}
to define factorization algebras on a topological space.
As we shall soon see, the topological exponential suffers one
drawback: it is not an exponential functor because
the functor
\( Y \mapsto X \times Y \) commutes with colimits only when
\( X \) is core-compact.

\subsubsection{First steps towards exponentiability}

\begin{proposition}
	\label{Continuous bijection}
	Let \( \{X_i\}_{i \in I} \) be a small family of topological
	spaces.
	The bijection
	\[
		\textstyle
		\TopExp\big(\coprod_{i \in I} X_i\big)
		\longrightarrow
		\SmallFiniteProduct {i \in I} \TopExp(X_i)
	\]
	is continuous.
\end{proposition}

\begin{proof}
	To show that this map is continuous it is enough to 
	check that its composition with the projections
	\( (\amalg_{i \in I} X_i)^K \to \TopExp(\amalg_{i \in I} X_i) \)
	is continuous for every finite set \( K \).
	For such a \( K \), the space
	\( (\amalg_{i \in I} X_i)^K \) is a disjoint union of
	spaces of the form \( \prod_{j \in J} X_j^{K_j} \)
	with \( J \subset I \) finite, and each
	projection map
	\( \prod_{j \in J} X_j^{K_j}
	\to \prod_{j \in J} \TopExp(X_j)
	\to \SmallFiniteProduct{i \in I} \TopExp(X_i) \)
	is continuous.
\end{proof}

\begin{lemma}[{\cite[2.4 \& 2.5]{hjm:28-4}}]
	\label{Lemma: perfect}
	Given a separated space \( X \), the projection map
	\[
		X^n \longrightarrow \TopExp(X)
	\]
	factors as a composite
	\[
		X^n \longrightarrow
		\TruncatedTopExp n (X) \subset \TopExp(X)
	\]
	of a closed quotient map followed by a closed embedding,
	for every \( n \geq 0 \).
\end{lemma}

\begin{lemma}
	\label{lemma: truncated is ok}
	For any small family of separated topological
	spaces
	\( \{X_i\}_{i \in I} \) and every
	natural \( n \), the canonical map
	\[
		\textstyle
		\SmallFiniteProduct{i \in I}
		\TruncatedTopExp n (X_i)
		\longrightarrow \TopExp\big(\coprod_{i \in I} X_i\big)
	\]
	is continuous.
\end{lemma}

\begin{proof}
	By definition of the finite product, it is enough to show it
	for \( I \) finite.
	Since each \( X_i \) is separated, the quotient map
	\( X_i^n \twoheadrightarrow \TruncatedTopExp n (X_i) \)
	is perfect \UnskipRef{Lemma: perfect},
	and hence a perfect map.
	Thus the product
	\( \prod_{i \in I} X_i^n
	\to \prod_{i \in I} \TruncatedTopExp n (X_i) \)
	is again a perfect map, and, in particular, a quotient map.
	Then, the map \( \prod_{i \in I} X_i^n
	\to (\amalg_{i \in I}X_i)^{I \sqcup \dots \sqcup I}
	\to \TopExp(\amalg_{i \in I} X_i) \)
	is continuous and by the previous
	observation, factors as a continuous map through the quotient
	\( \prod_{i \in I}\TruncatedTopExp n(X_i) \).
\end{proof}

\subsubsection{The topological exponential is not an exponential}

\begin{lemma}
	The topological exponential \( \TopExp(\Circle) \) contains
	a copy of the infinite bouquet of circles
	\( \vee^{\OrdinalOmega} \Circle \).
\end{lemma}

\begin{proof}
	We shall build a sequence of closed embeddings
	\[
		\begin{tikzcd}
			\Circle 
			\rar[hook]
			\dar[hook]
				& \dots
				\rar[hook]
				\dar[hook]
					& \vee^n_{i = 1} \Circle
					\rar[hook]
					\dar[hook]
						& \dar[hook] \dots \\
			\TruncatedTopExp 1(\Circle)
			\rar[hook]
				& \dots
				\rar[hook]
					& \TruncatedTopExp n(\Circle)
					\rar[hook]
						& \dots
		\end{tikzcd}
	\]
	which shall lead to a closed embedding
	\( \vee^\OrdinalOmega \Circle
	\hookrightarrow \TopExp(\Circle) \).
	For this, we embed two circles
	into the torus \( \Torus 2 \)
	via the vectors \( (0,1) \) and \( (1,1) \), three
	circles in \( \Torus 3 \)
	via the vectors \( (0,0,1) \), \( (0,1,1) \) and \( (1,1,1) \)
	etc.
	This defines a compatible family of closed embeddings
	\( \vee_{i =1}^n \Circle \hookrightarrow \Torus n \).
	Since moreover the projection map
	\( \Torus n \to \TruncatedTopExp n (\Circle) \) is closed,
	we get continuous closed maps
	\( \vee_{i=1}^n \Circle \to \TruncatedTopExp n (\Circle) \).
	By construction, they are injective and fit as expected
	in the diagram of closed embeddings above.
\end{proof}

\begin{theorem}[(The topological exponential is not an exponential)]
	\label{theorem: not an exponential}
	The canonical continuous bijection
	\[
		\TopExp\left(\Rationals \amalg \Circle\right)
		\longrightarrow
		\TopExp(\Rationals) \times \TopExp\left(\Circle\right)
	\]
	is not a homeomorphism.
\end{theorem}

\begin{proof}
	Using the same tori embeddings as in the previous
	lemma, one can fit a copy of
	\( \Rationals \times \vee_{i = 1}^n \Circle \)
	in \( \TruncatedTopExp {n+1}(\Rationals \amalg \Circle) \)
	for every \( n \geq 1 \).
	One can then check that the continuous bijection
	\( \TopExp(\Rationals \amalg \Circle) \to
	\TopExp(\Rationals) \times \TopExp(\Circle) \)
	restricts to the continuous bijection
	\( \varinjlim_{0 < n < \OrdinalOmega}  \Rationals \times
		\vee_{i = 1}^n \Circle
		\to \Rationals \times \vee^{\OrdinalOmega} \Circle
	\)
	which is not a homeomorphism
	\cite[3.2]{arXiv:2105.12220}.
\end{proof}

\subsubsection{The topological exponential is almost an exponential}

As we have just seen, the canonical continuous bijection
\( \TopExp(X \amalg Y) \to \TopExp(X) \times \TopExp(Y) \) is not
always a homeomorphism.
However when \( X \) and \( Y \) are separated,
its inverse is still sequentially continuous.

\begin{remark}[(Converging sequences in a colimit topology)]
	\label{Remark: converging sequence in colimit topology}
	Let \( Z_0 \hookrightarrow
	\cdots \hookrightarrow Z_p \hookrightarrow \cdots \)
	be a sequence of closed embeddings
	between \( \TOne \) topological
	spaces and let \( Z \) denote its colimit.
	Then every morphism \( K \to Z \) with \( K \) compact
	factors through one \( Z_p \subset Z \)
	\cite[2.4.2]{isbn:9780821843611}.
	More generally this is true if \( Z \) is the colimit of an
	ordinal sequence of closed embeddings.

	As a consequence, if \( X \) is separated, a sequence
	\( (S_n)_{n \in \Naturals} \) in \( \TopExp(X) \),
	converges only if the sequence of cardinals
	\( |S_n|_{n < \OrdinalOmega} \) is bounded.
\end{remark}

\begin{proposition}
	\label{Proposition: almost an exponential}
	Let \( \{X_i\}_{i \in I} \) be a small family of separated spaces,
	the canonical bijection
	\[
		\textstyle
		\SmallFiniteProduct{i \in I} \TopExp(X_i)
		\longrightarrow \TopExp\big(\coprod_{i \in I} X_i\big)
	\]
	is sequentially continuous.
\end{proposition}

\begin{proof}
	Let \( S \) be a sequence in \( \TopExp(\amalg_{i \in I} X_i) \).
	Since \( S_n \) is a finite subset of \( \amalg_{i \in I} X_i \)
	for each natural number \( n \),
	the union \( \cup_{n < \OrdinalOmega} S_n \)
	intersects only a countable number of \( X_i \),
	and we can thus reduce to the case where \( I \) is countable.
	Given a sequence \( X_0, X_1, \dots \) of separated spaces,
	the sequence
	\[
		\TopExp(X_0) \hookrightarrow
		\TopExp(X_0) \times \TopExp(X_1)
		\hookrightarrow
		\TopExp(X_0) \times \TopExp(X_1) \times \TopExp(X_2)
		\hookrightarrow
		\dots
	\]
	is made of closed embeddings between separated spaces.
	Thus \( \cup_{n < \OrdinalOmega} S_n \) intersects only a finite
	number of \( X_i \)
	\UnskipRef{Remark: converging sequence in colimit topology}.

	Then, we only need to consider the case of two separated spaces
	\( X \) and \( Y \).
	In that case,
	the sequence \( S \) is then comprised of a pair of two sequences
	\( S(X) \) and \( S(Y) \).
	Because \( \TopExp(X) \) is a union of closed embeddings between
	separated spaces, \( S(X) \) is bounded in cardinality
	\UnskipRef{Remark: converging sequence in colimit topology}.
	The same is true for \( S(Y) \).
	We conclude using that
	\(
		\TruncatedTopExp n (X) \times \TruncatedTopExp n (Y)
		\longrightarrow \TopExp(X \amalg Y)
	\)
	is continuous for every \( n \)
	\UnskipRef{lemma: truncated is ok}.
\end{proof}

\subsubsection{The topological exponential is a restricted exponential}
 
As we have explained earlier,
the topological exponential is not an exponential
functor because the functor \( Y \mapsto X \times Y \) does not
commute with colimits in general, for a given \( X \).
One might ask whether the exponential property could still hold
if restricted to core-compact spaces, i.e., the spaces \( X \) for
which \( Y \mapsto X \times Y \) commutes with colimits.
The answer to this question is non-obvious as \( \TopExp(X) \) is
usually not going to be core-compact, even when \( X \) is.

\begin{proposition}
	Let \( \{X_i\}_{i \in I} \) be a small family of separated and
	core-compact topological spaces.
	The canonical bijection
	\[
		\textstyle
		\SmallFiniteProduct{i \in I} \TopExp(X_i)
		\longrightarrow
		\TopExp\big(\coprod_{i \in I} X_i\big)
	\]
	is a homeomorphism.
\end{proposition}

\begin{proof}
	We only need to show that the above map is continuous
	\UnskipRef{Continuous bijection}.
	By definition of the finite product, we can reduce to the case
	of a finite \( I \).
	Since each \( X_i \) is separated, the projection map
	\( X_i^n \to \TruncatedTopExp n (X_i) \) is a perfect map
	and thus \( \TruncatedTopExp n (X_i) \) is core-compact.
	Because sequential unions of core-compact spaces commute with
	finite products
	\cite{arXiv:2105.12220}, one has
	canonical homeomorphisms
	\[
		\textstyle
		\prod_{i \in I} \TopExp(X_i)
		\IsCanonicallyIsomorphicTo
		\prod_{i \in I} \varinjlim_{n \in \PointedOmega}
		\TruncatedTopExp n (X_i)
		\IsCanonicallyIsomorphicTo
		\varinjlim_{n \in \PointedOmega} \prod_{i \in I}
		\TruncatedTopExp n (X_i)
	\]
	and the map
	\[
		\textstyle
		\prod_{i \in I} \TopExp(X_i)
		\IsCanonicallyIsomorphicTo
		\varinjlim_{n \in \PointedOmega} \prod_{i \in I}
		\TruncatedTopExp n (X_i)
		\longrightarrow \TopExp\big(\amalg_{i \in I} X_i\big)
	\]
	is continuous, as a colimit of continuous maps
	\UnskipRef{lemma: truncated is ok}.
\end{proof}

\begin{corollary}
	Let \( X \) be a separated and core-compact topological
	space. Then
	\( \TopExp(X) \) is the free idempotent commutative topological
	monoid on \( X \).
\end{corollary}

\subsection{The metric exponential}

In addition to not being an exponential functor, the topological
exponential also does not preserve metrizability.
In fact \( \TopExp(X) \) is almost never metrizable.

\begin{proposition}
	Let \( X \) be a metrizable topological space.
	If \( X \) has an accumulation point,
	then \( \TopExp(X) \)
	is not metrizable.
\end{proposition}

\begin{proof}
	Pick a metric \( D \) inducing the topology on \( \TopExp(X) \).
	Let \( x \in X \) be an accumulation point.
	For every \( n \geq 0 \), using the quotient map
	\( X^n \to \TruncatedTopExp n (X) \), one can find a subset
	\( S_n \subset X \) made of exactly \( n \) elements such that
	\( D(S_n, \{x\}) \leq 1/n \).
	In other words, \( S_n \to_{n \to +\infty} \{x\} \)
	and \( |S_n| \to_{n \to +\infty} +\infty \) which is forbidden
	\UnskipRef{Remark: converging sequence in colimit topology}.
\end{proof}

One way to remedy this is to compute the colimit defining
the exponential not in the category of topological spaces but rather
in the category \( \GenMet \) of generalized metric spaces.

A generalized metric space is a metric space whose distance function
is allowed to have the value \( + \infty \).
Morphisms in \( \GenMet \) are the metric maps: the maps
\( f \From (M, d_M) \to (N, d_N) \) such that
\( d_N(f(x), f(y)) \leq d_M(x, y) \) for every \( x, y \in M \).

The main advantage of the category of generalized metric spaces is that
it admits all small limits and colimits
\cite[4.5(3)]{doi:10.1017/jsl.2016.39}.
Computing the colimit defining the exponential
in \( \GenMet \), one obtains the \emph{metric exponential}, of which
we give a concrete definition.

\begin{definition}[(Metric exponential of a metric space)]
	Given a (generalized) metric space \( (X, d) \),
	its metric exponential is the generalized metric space
	\( (\Exp(X), D) \), 
	where
	\[
		D(S, T) \coloneqq
		\max
			\begin{cases}
				\max_{s \in S} \min_{t \in T} d(s,t)
				\\
				\max_{t \in T} \min_{s \in S} d(s, t).
			\end{cases}
	\]
	We shall denote the metric exponential by \( \MetricExp(X) \).
	In particular, one has
	\( D([\emptyset], T) = D(T, [\emptyset]) = + \infty \)
	when \( T \) is not empty.
\end{definition}

\begin{remark}
	The metric subspace \( \MetricRan(X) \subset \MetricExp(X) \)
	is used by Lurie as an intermediate tool
	to deal with locally constant non-unital factorization algebras
	which are locally constant cosheaves on \( \TopRan(X) \)
	\cite[3.3.2]{arXiv:0911.0018}.
	In a remark in \emph{Higher Algebra},
	he suggests using a variant of the exponential
	\( \MetricExp(X) \)
	where \( D([\emptyset], T) = D(T, [\emptyset]) = 0 \)
	for every \( T \subset X \), to deal with unital factorization
	algebras.

	The metric exponential has also been used by Knudsen
	\cite{doi:10.2140/gt.2018.22.4013} in his work extending the
	constructions of Francis and Gaitsgory
	\cite{doi:10.1007/s00029-011-0065-z}
	to the topological setup.
\end{remark}

The topology of the metric exponential
admits a basis given by opens of the form \( [U_i]_{i \in I} \)
where
\[
	S \in [U_i]_{i \in I} \quad \Leftrightarrow\quad
	\forall i \in I, \quad
	S \cap U_i \neq \emptyset.
\]
This allows us to define the metric exponential \( \MetricExp(X) \)
when \( X \) is only a topological space.

\begin{definition}[(Metric exponential of a topological space)]
	For a topological space \( X \), its metric
	exponential \( \MetricExp(X) \) consists of
	the set \( \Exp(X) \) endowed with the coarsest
	topology including all \( [U_i]_{i \in I} \)
	for every finite set \( I \) of open subsets \( U_i \subset X \).

	This definition is functorial: if \( f \From X \to Y \)
	is a continuous map, the preimage of
	\( [U_i]_{i \in I} \) by \( \Exp(f) \)
	equals \( [f^{-1}(U_i)]_{i \in I} \).
\end{definition}

Before looking at the exponential property of \( \MetricExp \), we shall
discuss how some limits and colimits are computed in
\( \GenMet \).
Given a small family \( \{(X_i, d_i)\}_{i \in I} \) of
(pointed) metric spaces, their
\begin{description}
	\item[coproduct]
		is the disjoint union of
		sets \( \amalg_{i \in I} X_i \)
		endowed with the distance \( d \) for which
		\[
			d(x_i, y_j) = \begin{cases}
				d_i(x_i, y_i), \quad \text{if}
				\quad i = j
				\\
				+\infty, \quad \text{if}
				\quad i \neq j
			\end{cases}
		\]
	\item[product]
		is the product set
		\( \prod_{i \in I} X_i \) endowed with the sup
		metric
		\[
			\textstyle
			d(\{x_i\}, \{y_i\})
			\coloneqq \sup_{i \in I} d_i(x_i, y_i)
		\]
	\item[finite product] is the finite product set
		endowed with the sup metric.
		In other words, in that case, the natural map
		\[
			\textstyle
			\SmallFiniteProduct{i \in I} X_i
			\longrightarrow \prod_{i \in I} X_i
		\]
		is an isometric embedding.
\end{description}

\begin{proposition}[(Exponential property)]
	The metric exponential \( \MetricExp \) is an
	exponential functor
	for both \( \GenMet \) and \( \Top \).
\end{proposition}

\begin{proof}
	Starting with the metric case:
	let \( \{(X_i, d_i)\}_{i \in I} \) be a small family of
	metric spaces.
	We only need to show that the bijections in the exponential
	structure of \( \Exp \) are isometric.
	Let \( S \) and \( T \) be two finite subsets of the union
	in \( (X,d) \coloneqq \amalg_{i \in I} (X_i,d_i) \)
	and write \( S_i \coloneqq S \cap X_i \) and
	\( T_i \coloneqq T \cap X_i \) for every \( i \in I \).
	By construction of the disjoint union,
	if \( s \in S \) and \( t \in T \) do not
	belong to the same component \( X_i \), their distance
	\( d(s, t) \) in \( X \) is infinite.
	As a consequence the distance \( D(S, T) \)
	in \( \MetricExp(X) \) becomes
	\[
		D(S, T) =
		\max
			\begin{cases}
				\max_{s \in S} \min_{t \in T} d(s,t)
				=
				\sup_{i \in I} \max_{s \in S_i}
				\min_{t \in T_i} d_i(s,t)
				\\
				\max_{t \in T} \min_{s \in S} d(s, t)
				=
				\sup_{i \in I} \max_{t \in T_i}
				\min_{s \in S_i} d_i(s,t)
			\end{cases}
	\]
	and thus \( D(S,T) = \sup_{i \in I} D_i(S_i, T_i) \).

	Let \( \{X_i\}_{i \in I} \) be a small family of
	topological spaces.
	Because finite sets can only intersect a finite number of
	connected components, the open sets of the form
	\( [U_{j,k}]_{j \in J, k \in K_j} \) where \( J \subset I \)
	and each \( K_j \) are finite,
	and where each \( U_{j, k} \subset X_j \) is open, form a
	basis of the topology of \( \MetricExp(\amalg_{i \in I} X_i) \).
	It corresponds bijectively to
	the base open set
	inside \( \SmallFiniteProduct{i \in I} \MetricExp(X_i) \)
	given by \( \prod_{j \in J} [U_{j,k}]_{k \in K_j} \).
	Thus the bijection
	\( \MetricExp(\amalg_{i \in I} X_i)
	\to \SmallFiniteProduct{i \in I} \MetricExp(X_i) \)
	is a homeomorphism.
\end{proof}

\begin{proposition}
	Let \( X \) be a (generalized) metric space.
	Then \( \MetricExp(X) \) is the free idempotent commutative
	metric monoid on \( X \).
\end{proposition}

\begin{proof}
	The canonical map \( X \to \MetricExp(X) \) is
	an isometry by construction.

	So the only thing to show is that for
	\( (A, d) \)
	an idempotent and commutative metric monoid, the map
	\( \MetricExp(A) \to A \) sending \( S \subset A \) to
	\( \prod_{s \in S} s \in A \) --- which is well defined because
	\( A \) is commutative and is a monoid map because \( A \)
	is idempotent --- is a metric map.

	Given two finite subsets \( S, T \subset A \), we need to show
	that \( d(\prod_{s \in S} s, \prod_{t \in T} t)
	\leq D(S, T) \).
	If \( S \) or \( T \) is empty, it is immediate.
	Because \( A \) is an idempotent metric monoid
	\( d(a, bc) = d(aa, bc)
	\leq \max(d(a,b), d(a,c)) \) for every \( a, b, c
	\in A \).
	By straightforward induction, one gets the case where either
	\( S \) or \( T \) has a unique element.
	Let \( n \) be an integer and assume that
	the inequality has been shown for every \( S, T \) with
	\( |S| + |T| \leq n \).
	Let \( S, T \subset A \) with \( |S| + |T| = n+1 \).
	Without loss of generality, we can assume that there exists
	\( x \in S \) such that \( D(S, T) = d(x, T) \).
	Let \( S_0 \) denote the complement of \( x \) in \( S \).
	Then
	\[
		\begin{aligned}
			\textstyle
			d(\prod_{s \in S} s, \prod_{t \in T} t)
				&= 
				\textstyle
				d(x \times \prod_{s \in S_0} s, \prod_{t \in T} t)
			\\
			  &\leq 
			  \textstyle
			  \max\big(d\big(x, \prod_{t \in T} t\big),
			  d\left(\prod_{s \in S_0} s, \prod_{t \in T} t\right)\big)
			  & & \text{(\( A \) is metric)}
			\\
			  &\leq 
			  \textstyle
			  \max(d(x, T),
			  D(S_0, T))
			  & & \text{(by hypothesis)}
			\\
			  &=
			  D(S, T)
			  & & \text{(by definition of \( x \))}
		\end{aligned}
	\]
	ending showing that \( \MetricExp (A) \to A \) is a metric map.
\end{proof}

\begin{proposition}
	For every topological space \( X \), the identity
	\[
		\TopExp(X) \longrightarrow \MetricExp(X)
	\]
	is a continuous map, which restricts to homeomorphisms
	\[
		\TruncatedTopExp n (X) = \TruncatedMetExp n (X)
	\]
	for every \( n \in \PointedOmega \),
	whenever \( X \) is separated.
\end{proposition}

\begin{proof}
	Let \( U \subset X \) be an open subset.
	Let \( n \geq 1 \) be an integer
	and let \( \sigma \) denote the permutation \( (1\cdots n) \).
	Then the preimage along \( X^n \to \Exp(X) \) of \( [U] \)
	is the set \( \cup_{i \leq n} \sigma^i (U \times X^{n-1}) \)
	which is open.
	It follows that \( [U] \) is open in \( \TopExp(X) \).
	For a finite \( I \),
	\( [U_i]_{i \in I} = \cap_{i \in I} [U_i] \) is then also open
	in \( \TopExp(X) \).
\end{proof}

\subsection{The minimal exponential}

\begin{definition}
	Given a topological space, the \emph{minimal exponential}
	\( \ExpMin(X) \)
	is the set \( \Exp(X) \) endowed with the coarsest topology
	containing the subsets \( \Exp(U) \subset \Exp(X) \)
	for all open subsets \( U \subset X \).
\end{definition}

\begin{remark}
	One distinctive feature of the minimal exponential is that
	the point presenting the empty configuration \( [\emptyset] \)
	is dense.
\end{remark}

Families of open subsets \( \{U_i \subset V \}_{i \in I} \)
for which
\( \{\Exp(U_i) \subset \Exp(V)\}_{i \in I} \) is a cover
in the minimal exponential, were introduced by Weiss in his work
on the embedding calculus
\cite{doi:10.2140/gt.1999.3.67}.
This notion of covering is used
by Costello and Gwilliam
to define factorization algebras in general
\cite[1.4.1]{doi:10.1017/9781316678626}.
It is also used by Ayala and Francis in their study of factorization
homology
\cite[2.6]{arXiv:1903.10961}.

\begin{proposition}[(Exponential property)]
		The minimal exponential is an exponential
		functor on the category of topological spaces.
\end{proposition}

\begin{proof}
	Let \( \{X_i\}_{i \in I} \) be a small family of spaces.
	Since \( [\emptyset] \) is open in the minimal topology,
	the finite product of the \( \ExpMin(X_i) \) is a subspace
	of the product endowed with the box topology.

	Given a family of open subsets
	\( \{U_i \subset X_i\}_{i \in I} \), one has bijections
	\[
		\textstyle
		\Exp\big(\coprod_{i \in I} U_i\big)
		\IsCanonicallyIsomorphicTo
		\SmallFiniteProduct{i \in I} \Exp(U_i)
		\IsCanonicallyIsomorphicTo
		\big(\prod_{i \in I} \Exp(U_i)\big)
		\cap \left(\SmallFiniteProduct{i \in I} \Exp(X_i)\right)
	\]
	showing the correspondence between the two bases of open
	sets between \( \ExpMin(\amalg_{i \in I} X_i) \)
	and \( \SmallFiniteProduct{i \in I} \ExpMin(X_i) \).
\end{proof}

\begin{remark}[(Minimality)]
	\label{minimality}
	The functors \( \TopExp \), \( \MetricExp \)
	and \( \ExpMin \) preserve open embeddings between topological
	spaces.
	In the category of exponential functors of
	\( \Top \) having this preservation property,
	\( \ExpMin \) is a final object.
\end{remark}

\subsection{Weak homotopy type of the exponentials}

The functors \( X \mapsto \TruncatedTopExp n (X) \) have
interesting homotopy properties as shown by Handel.
In particular, he showed that for \( X \) a separated and path
connected space, \( \TopRan (X) \) is weakly contractible
\cite[4.3]{hjm:28-4}.
Curtis \& Nhu showed that \( \MetricRan(X) \)
is \emph{homeomorphic} to a linear space, whenever
\( X \) is a
connected, locally path connected
metric space, which is a countable union of finite
dimensional compact spaces
\cite{doi:10.1016/0166-8641(85)90005-7}; it is in particular
\emph{contractible} in the strong sense.
In a simpler proof, Lurie
showed that \( \MetricRan(M) \) is
weakly contractible when \( M \) is a connected manifold
\cite[3.3.6]{arXiv:0911.0018}.

Here we shall enhance these results by describing the weak homotopy type
of each exponential for any separated and locally path connected space.
Since \( [\emptyset] \) is dense in \( \ExpMin(X) \),
it follows that \( \ExpMin(X) \) is contractible for any space \( X \).
Hence, we shall focus on the metric and the topological exponentials.
We start with a lemma due to Beilinson and Drinfeld.

\begin{lemma}
	\label{lemma: trivial group}
	Let \( G \) be a group endowed with an extra
	operation
	\( \wedge \From G \times G \to G \) such that
	\( \wedge \) is associative, idempotent and such that
	\( ab \wedge cd = (a \wedge c)(b \wedge d) \) for any
	\( a, b, c, d \in G \).
	Then \( G \) is a trivial group.
\end{lemma}

\begin{proof}
	For every \( g \in G \) one has
	\[
		g \wedge g = g \implies (1g) \wedge (1g) = g
		\implies (1 \wedge g)^2 = g.
	\]
	So for every \( h \in G \), letting \( g = 1 \wedge h \),
	\[
		(1 \wedge h)^2 = (1 \wedge 1 \wedge h)^2 = 1 \wedge h
	\]
	since \( G \) is group, we get \( 1 \wedge h = 1 \) and
	\( h = (1 \wedge h)^2 = 1 \) for every \( h \in G \).
\end{proof}

\begin{lemma}
	If \( X \) is path connected, then \( \TopRan(X) \)
	is path connected.
	As a consequence, \( \MetricRan(X) \) is also path connected.
\end{lemma}

\begin{proof}
	Given two proper finite subsets
	\( S, T \subset X \), there exists
	a large enough positive \( n \in \Naturals \) and two
	tuples \( (s_1, \dots, s_n) \) and \( (t_1, \dots, t_n)  \)
	representing respectively \( S \) and \( T \).
	Since \( X \) is path connected, there exists a path
	between those two tuples in \( X^n \) and since
	the map \( X^n \to \TopExp(X) \) is continuous by construction
	and factors through \( \TopRan(X) \),
	this gives us a continuous path between \( S \) and \( T \)
	in \( \TopRan(X) \).
\end{proof}

In what follows, let us denote by \( \IdemTwo \) the
commutative idempotent monoid with two elements
\( (\{0,1\}, \vee) \) and endow it with the discrete topology.

\begin{theorem}
	Let \( X \) be a locally path connected topological space.
	The monoid map
	\[
		\begin{tikzcd}
			\MetricExp(X)
			\rar["\exists"]
				&
				\displaystyle
				\bigoplus_{\PiZero(X)}
				\IdemTwo
		\end{tikzcd}
	\]
	sending a finite subset \( S \subset X \) to the family
	\( \{\exists_i\}_{i \in \PiZero(X)} \) with \( \exists_i = 0 \)
	if and only if no element of \( S \) belongs to the connected
	component \( X_i \subset X \), is continuous and a weak
	homotopy equivalence.

	Moreover, if \( X \) is also separated,
	the induced continuous map
	\[
		\begin{tikzcd}
			\TopExp(X)
			\rar["\exists"]
				&
				\displaystyle
				\bigoplus_{\PiZero(X)}
				\IdemTwo
		\end{tikzcd}
	\]
	is also a weak homotopy equivalence.
\end{theorem}

\begin{proof}
	Using that \( \MetricExp \) is an exponential and the
	fact that for each connected component
	\( X_i \subset X \),
	\( \MetricExp(X_i) \) is the disjoint union of
	\( [\emptyset] \) and \( \MetricRan(X_i) \),
	\( \MetricExp(X) \) splits as
	\[
		\textstyle
		\MetricExp(X)
		\IsCanonicallyIsomorphicTo
		\SmallFiniteProduct{i \in \PiZero(X)} \MetricExp(X_i)
		\IsCanonicallyIsomorphicTo
		\coprod_{\substack{J \subset \PiZero(X) \\
		J \text{ finite}}} \prod_{j \in J} \MetricRan(X_j)
	\]
	immediately showing that the map \( \exists \) is continuous
	and that \( \PiZero(\exists) \) is a bijection.
	Let \( S \subset X \) be a finite subset. Then, since
	\( S \cup S = S \), the monoid structure of \( \MetricExp(X) \)
	induces an associative and idempotent map
	\[
		\PiN n (\MetricExp(X), S) \times
		\PiN n (\MetricExp(X), S)
		\longrightarrow
		\PiN n (\MetricExp(X), S)
	\]
	which satisfies the exchange property, for every \( n > 0 \).
	As a consequence each of these groups is trivial
	\UnskipRef{lemma: trivial group}.

	When \( X \) is separated, the canonical bijection
	\[
		\textstyle
		\TopExp(X)
		\longrightarrow
		\SmallFiniteProduct{i \in \PiZero(X)} \TopExp(X_i)
	\]
	is sequentially continuous with continuous inverse
	\UnskipRef{Proposition: almost an exponential}
	and thus one has
	\[
		\textstyle
		\PiZero(\TopExp(X))
		\IsCanonicallyIsomorphicTo
		\PiZero\left(\SmallFiniteProduct{i \in \PiZero(X)}
		\TopExp(X_i)\right)
	\]
	because the segment \( [0,1] \) is a sequential space.
	Since spheres and balls are also sequential spaces,
	the sequentially continuous map
	\( \TopExp(X) \times \TopExp(X) \to \TopExp(X) \)
	still induces maps
	\[
		\PiN n (\MetricExp(X), S) \times
		\PiN n (\MetricExp(X), S)
		\longrightarrow
		\PiN n (\MetricExp(X), S)
	\]
	so using the same proof as for the metric case, we see that
	\( \exists \From \TopExp(X) \to \oplus_{i \in \PiZero(X)}
	\IdemTwo \) is a weak equivalence.
\end{proof}

\section{Interlude: minimal enclosing balls}

Given a normed vector space \( V \) and a
proper finite subset \( S \subset V \),
a minimal enclosing ball for \( S \) is
a closed ball \( B \subset V \) which
contains \( S \) and such that any other
ball containing \( S \) have a bigger radius.

\begin{figure}
	\begin{tikzpicture}
		\draw[very thick] (0,0) circle (2cm);
		\draw[fill] (60:2) circle (1mm);
		\draw[fill] (172:2) circle (1mm);
		\draw[fill] (276:2) circle (1mm);
		\draw[fill] (46:1.7) circle (1mm);
		\draw[fill] (143:1.4) circle (1mm);
		\draw[fill] (243:0.9) circle (1mm);
	\end{tikzpicture}
	\caption{The enclosing circle of a finite set of
	points in the plane.}
\end{figure}
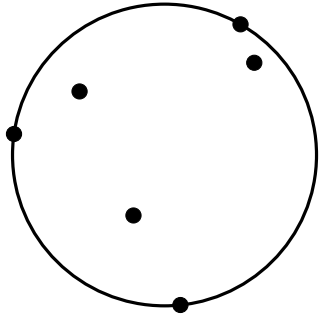

Using classic convex optimization results, one can show
the existence and uniqueness of a minimal
enclosing ball in the case of rotund reflexive normed vector spaces
\cite{doi:10.1007/s11590-012-0483-7}.
One may wonder whether the center \( \Center_S \) and the radius
\( \Radius_S \) of the minimal enclosing ball of a proper
finite subset \( S \) vary continuously with \( S \).
This question is naturally posed using the metric exponential
\( \MetricExp(V) \).

For such a general space as a reflexive vector space, one can
only show that the center \( \Center_S \) varies continuously
with \( S \) \emph{for the weak topology} of \( V \).
The continuity of the center becomes strong if one instead
considers a restricted version of the minimal enclosing ball
problem.
This is what we shall see here.

\begin{definition}[(Restricted minimal enclosing ball)]
	Let \( V \) be a normed vector space and let \( S \subset V \)
	be a proper finite subset of \( V \).
	A restricted minimal enclosing ball
	is a closed ball \( B \subset V \) containing \( S \) and
	whose center belongs to the convex hull \( \Conv(S) \)
	of \( S \), such that, any other ball with center in
	\( \Conv(S) \) and containing \( S \) have a bigger radius.
\end{definition}

\begin{remark}
	In a Hilbert space \( H \),
	the restricted minimal enclosing ball
	of \( S \subset H \)
	coincides with its minimal enclosing ball.
\end{remark}

Restricted minimal enclosing balls might not be unique for a given
norm.
We shall then restrict our attention to spaces that can be endowed
with a norm with strictly convex unit ball.

\begin{definition}[(Rotund vector space)]
	We shall say that a topological vector space is
	rotund if its topology can be induced by a norm
	for which the closed unit ball is strictly convex: the
	equation
	\[
		\norm{x}= \norm{y}= \norm{\frac{x+y}{2}}
	\]
	holds only when \( x = y \).
	By extension, we shall say that such a norm \emph{is rotund}.
\end{definition}

\begin{example}
	Finite dimensional vector spaces are rotund.
	More generally separable complete normable spaces
	are rotund
	\cite[Thm. 9]{doi:10.1090/s0002-9947-1936-1501880-4}.
	The space \( \lInfinity \) is not rotund
	\cite[Thm. 8]{doi:10.1090/s0002-9947-1955-0067351-1}.
	Every reflexive normed vector space is rotund
	\cite[{Cor. 1 (i)}]{doi:10.1090/s0002-9904-1966-11606-3}.
\end{example}

\begin{lemma}
	Let \( V \) be a normed vector space.
	The correspondence sending
	\( S \in \MetricRan(V) \) to its convex
	hull \( \Conv(S) \subset V \) is
	continous.
\end{lemma}

\begin{proof}
	It is upper hemicontinuous: let \( v \in V \) and let
	\( \epsilon > 0 \).
	One has
	\[
		\Conv^\mathrm{u}(\Ball(v, \epsilon))
		\coloneqq
		\{ S \mid \Conv(S) \subset \Ball(v, \epsilon) \}
		= \Ball(\{v\}, \epsilon)
	\]
	showing that the upper inverse image preserves opens.

	It is lower hemicontinuous: let
	\[
		S \in \Conv^\mathrm{l}(\Ball(v, \epsilon))
		\coloneqq
		\{ S \mid \Conv(S) \cap \Ball(v, \epsilon)
		\neq \emptyset \}
	\]
	then there exists \( s \in S \) such that
	\( \norm{s - v} < \epsilon \).
	Let \( \delta = \epsilon - \norm{s-v} \), then
	\( \Ball(S, \delta)
	\subset \Conv^\mathrm{l}(\Ball(v, \epsilon)) \)
	showing that the lower inverse image preserves opens.
\end{proof}

\begin{lemma}
	The function
	\[
		\begin{tikzcd}[row sep = tiny]
			V \times \MetricRan(V) \rar & \Reals_+ \\
			(v, S) \rar[mapsto] & \max_{s \in S} \norm{v-s}
		\end{tikzcd}
	\]
	is continuous.
\end{lemma}

\begin{proof}
	Consider a converging sequence
	\( (v_n, S_n) \to_{n \to ∞} (v, S) \).
	For \( \epsilon > 0 \) small enough, if \( S_n \) is
	at distance less than \( \epsilon \) from \( S \), then
	\( S_n \) must have more points that \( S \) and for each
	\( x \in S_n \), there is a unique \( s_x \in S \) such that
	\( \norm{x-s_x} \leq \epsilon \).
	This gives us a partition of \( S_n \) as
	\( S_n = \cup_{s \in S} S_n(s) \).
	Then for \( \norm{v_n -v} \leq \epsilon \)
	and \( D(S_n, S) \leq \epsilon \) one has
	\begin{align*}
		\Big|\max_{s \in S} \norm{v-s}
		- \max_{t \in S_n} \norm{v_n - t}\Big|
		&
		\leq
		\max_{s \in S} \Big|\norm{v-s} - \max_{t \in S_n(s)}
		\norm{v_n - t}\Big| \\
		&
		\leq
		\max_{s \in S}\left(\norm{v-v_n} + 
		\max_{t \in S_n(s)} \norm{s - t}\right) \\
		& \leq 2 \epsilon
	\end{align*}
	showing that \( (v, S) \mapsto
	\max_{s \in S} \norm{v-s} \) is continuous.
\end{proof}

\begin{theorem}[(Solution to the restricted
	minimal enclosing ball problem)]
	Let \( V \) be a vector space endowed with a rotund norm.
	Then every proper
	finite subset \( S \subset V \) admits
	a restricted minimal enclosing ball of radius
	\[
		\Radius_S \coloneqq
		\inf_{v \in \Conv(S)} \max_{s \in S} \norm{s-v}
	\]
	and this ball is unique.

	Moreover, the function
	\[
		\MetricRan(V) \longrightarrow V \times \Reals_+
	\]
	mapping a proper finite subset \( S \subset V \)
	to the pair \( (\Center_S, \Radius_S) \)
	where \( \Center_S \) denotes the center of the
	restricted
	minimal enclosing ball of \( S \), is continuous.
\end{theorem}

\begin{proof}
	Because \( S \) is finite, its convex hull
	\( \Conv(S) \) is compact in \( V \). As a consequence
	\( \Radius_S \) is finite and \( S \) admits
	a restricted minimal enclosing ball.
	It is unique: because the norm is rotund,
	the function
	\( v \mapsto \max_{s \in S} \norm{s - v} \) is strictly
	convex, so its infimum on the convex hull \( \Conv(S) \)
	is attained at a unique point.

	The continuity of \( \Center_S \) and \( \Radius_S \) can be
	obtained using the Maximum Theorem
	\cite[17.31]{doi:10.1007/3-540-29587-9}:
	the correspondence \( S \mapsto \Conv(S) \) is continuous
	with compact values and 
	the function
	\( (v, S) \mapsto \max_{s \in S} \norm{s-v} \) is continuous
	by the previous lemmas.
\end{proof}

\section{Stratification of the exponentials}

The exponential of a set \( X \) admits a natural
counting function \( \Exp(X) \to \Naturals \) sending each
finite subset \( S \subset X \) to its cardinal \( |S| \).
In this section we study the exponentials endowed
with the stratification given by this counting function.
We shall show that, under some conditions on \( X \),
the metric exponential \( \MetricExp(X) \) is conically
stratified.
As a consequence the \( \infty \)-categories
of constructible hypersheaves on
\( \MetricExp(X) \) and \( \TopExp(X) \) are equivalent, which leads
to the following statement: locally constant factorization algebras
on \( X \) in the sense of Beilinson-Drinfeld are equivalent
to locally constant factorization algebras on \( X \)
in the sense of Lurie.

\subsection{The exponentials as stratified spaces}

There are several inequivalent definitions of stratified spaces.
The following one is a mild one, introduced by Lurie
\cite[A.5.1]{arXiv:0911.0018}.

\begin{definition}[(Stratified space)]
	\label{strat}
	A stratified space is the data of a poset \( P \),
	endowed with the topology whose open sets are the
	upward-closed subsets,
	and a continuous map \( f \From X \to P \).
	For \( p \in P \), we shall write \( X_p \) for the
	fiber \( f^{-1}(p) \).

	A morphism of stratified spaces is a commutative square
	\[
		\begin{tikzcd}[ampersand replacement=\&]
			X
			\arrow[r, ""]
			\arrow[d, "", swap]
			\& Y
			\arrow[d, ""] \\
			P
			\arrow[r, "", swap]
			\& Q
		\end{tikzcd}
	\]
	where the top map is continuous and the bottom map
	is a poset map.
\end{definition}

In our case, we select the poset \( \PointedOmega \).
Since the empty configuration is dense in \( \ExpMin(X) \), the
minimal exponential shall never be stratified over \( \PointedOmega \)
or even \( \OrdinalOmega \) as soon as \( X \) is not empty.
We shall thus only have a look at the two other exponentials.

\begin{proposition}
	When \( X \) is separated,
	the canonical maps
	\[
		\TopExp(X) \longrightarrow
		\MetricExp(X) \longrightarrow \PointedOmega
	\]
	are continuous.
	Moreover, one has homeomorphisms
	\[
		\TruncatedTopExp n (X) = \TruncatedMetExp n (X)
	\]
	for every \( n \in \PointedOmega \).
\end{proposition}

\begin{proof}
	We need to show that \( \TruncatedMetExp n (X) \) is a closed
	subset of \( \MetricExp(X) \) for every \( n \in \PointedOmega \).
	For \( n = 0 \) this is obvious.
	Let \( S \) be a non-trivial finite subset of \( X \).
	Since \( X \) is separated, one can find a disjoint family of
	open neighborhoods \( \{s \in U_s\}_{s \in S} \).
	Then \( [U_s]_{s \in S} \) becomes an open neighborhood of \( S \)
	in \( \MetricExp(X) \) which lies in the complement of
	\( \TruncatedMetExp n (X) \).

	When \( X \) is separated, Handel has shown that
	the opens
	of the form \( [U_i]_{i \in I} \cap \TruncatedTopExp n (X) \)
	form a basis
	of the topology of \( \TruncatedTopExp n (X) \)
	\cite[2.11]{hjm:28-4}, giving us the homeomorphism
	\( \TruncatedTopExp n (X) = \TruncatedMetExp n (X) \)
	for every \( n \in \PointedOmega \).
\end{proof}

\subsection{Cones and joins}

\begin{definition}[(Open cone)]
	For a topological space \( X \), the \emph{open cone}
	of \( X \) is the set 
		\[
			\Cone(X)
			\coloneqq \{ 0 \} \amalg (\Reals_+^\ast \times X)
		\]
	with topology defined as follows:
	a subset \( U \subset \Cone(X) \)
	is open if and only if
	\( U \cap (\Reals_+^\ast \times X) \) is open,
	and if \( 0 \in U \), then
	\( (0, \varepsilon) \times X \subset U \)
	for some positive real number \( \varepsilon \). 

	If \( X \) is stratified over a poset \( P \), then
	\( \Cone(X) \) is naturally stratified over
	the poset \( P^\triangleleft \) obtained from \( P \) by
	adding a new element smaller than every other element of
	\( P \).
\end{definition}

\begin{warning}
	One should not confuse the \emph{cone}
	on \( X \) with the \emph{collapsed rectangle}
	defined as the quotient
	\( \Reals_+ \times X / \{0\} \times X \).
	When \( X \) is compact and separated, the cone on
	\( X \) and the collapsed rectangle on \( X \) are
	homeomorphic. This is no longer true in the general
	case: the cone on the open interval
	\( (0,1) \) can be embedded in \( \Reals^2 \), whereas
	the collapsed rectangle on \( (0,1) \) is not
	metrizable.

	If \( (X,d) \) is a metric space, the topology of the cone
	\( \Cone(X) \) is metrizable by letting
	\( d((\lambda, x),(\mu, y) = \max(|\lambda - \mu|, d(x,y)) \)
	and by adding \( d(0, (\lambda, x)) = \lambda \).
\end{warning}

\begin{definition}[(Join)]
	Given two posets \( P \) and \( Q \), their
	join is the poset
	\[
		P \Join Q \coloneqq P \amalg (P \times Q) \amalg Q
	\]
	where one adds to the disjoint sum the additional relations
	\( p < (p, q) \) and \( q < (p, q) \) for every
	\( (p, q) \in P \times Q \).

	Let \( X \to P \) and \( Y \to Q \) be
	two stratified topological spaces.
	Their join \( X \Join Y \) is the set
	\[
		X \Join Y \coloneqq X \amalg (X \times (0,1) \times Y) \amalg Y
	\]
	where a basis of opens is given by the opens
	\( U \subset X \times (0,1) \times Y \) together with
	opens \( X \amalg X \times (0, \varepsilon) \times V \) with
	\( V \subset Y \) open, and
	opens \( W \times (\delta, 1) \times Y \amalg Y \) with
	\( W \subset X \) open.

	It is naturally stratified over \( P \Join Q \).
\end{definition}

\begin{warning}
		Similarly to what we just said about the cone,
		when \( X \) and \( Y \) are both separated and compact,
		the \emph{join}
		of \( X \) and \( Y \) is homeomorphic to
		the \emph{collapsed brick}
		\( X \times [0,1] \times Y / R \) where
		\( R \) is the relation identifying
		\( X \times \{0\} \times Y \sim X \)
		and
		\( X \times \{1\} \times Y \sim Y \).
		In general, this is no longer the case.
\end{warning}

\begin{proposition}
	\label{Product of cones}
	Let \( X \to P \) and \( Y \to Q \) be two stratified spaces.
	Then there is a homeomorphism
	\[
		\Cone(X) \times \Cone(Y) \IsIsomorphicTo
		\Cone(X \Join Y)
	\]
	over the canonical
	isomorphism \( P^\triangleleft \times Q^\triangleleft
	\IsCanonicallyIsomorphicTo (P \Join Q)^\triangleleft \).
\end{proposition}

\begin{proof}
	The map sends bijectively tuples
	\( (\lambda, (x, t, y)) \in \Cone(X \Join Y) \)
	to tuples
	\( ((\lambda t, x),(\lambda(1-t), y))
	\in \Cone(X) \times \Cone(Y) \) and obviously
	respects the isomorphism
	\( P^\triangleleft \times Q^\triangleleft
	\IsCanonicallyIsomorphicTo (P \Join Q)^\triangleleft \).
	Let us see why it is open: there are four different
	cases to look at.
	
	\emph{Case 1}: open neighborhoods of the tip
	of \( \Cone(X \Join Y) \). Let \( \varepsilon > 0 \), then
	the open \( \{0\} \amalg (0,\varepsilon) \times (X \Join Y) \)
	is mapped to the open
	\( (\{0\} \amalg (0, \varepsilon) \times X)
	\times (\{0\} \amalg (0, \varepsilon) \times Y) \).

	\emph{Case 2}: open neighborhoods of \( \Cone(X \Join Y) \)
	not containing the tip but including \( X \).
	An open
	of the form
	\( (\alpha, \beta)
	\times (X \amalg X \times (0, \varepsilon) \times V) \)
	with \( 0 < \alpha < \beta \)
	is mapped to the open
	\( (\{0\} \amalg (0, \varepsilon \alpha) \times X)
	\times ((1-\varepsilon)\alpha, (1-\varepsilon)\beta) \times V\).

	\emph{Case 3}: open neighborhoods of \( \Cone(X \Join Y) \)
	not containing the tip but including \( Y \).
	Confere supra.

	\emph{Case 4}: a general open of \( \Cone(X \Join Y) \).
	Let \( U \subset X \), \( V \subset Y \) opens,
	\( 0 < \alpha < \beta \), \( 0 \leq t < s \leq 1 \).
	Then \( (\alpha, \beta) \times U \times (t,s) \times V \)
	is mapped to the open \( (t\alpha, s\beta) \times U
	\times ((1-s)\alpha,(1-t)\beta) \times V \).

	Since the image of basis neighborhoods form a basis of
	neighborhoods of \( \Cone(X) \times \Cone(Y) \),
	one can see that it is a homeomorphism.
\end{proof}

\begin{remark}
	A very similar proposition has been given
	by Ayala, Francis and Tanaka
	using the collapsed rectangle instead of the cone and the
	collapsed brick instead of the join
	\cite[3.8]{doi:10.1016/j.aim.2016.11.032}.
	Of course,
	both propositions agree in the case where both \( X \) and \( Y \)
	are compact and separated.
\end{remark}

\subsection{Conical stratification}

There are many inequivalent notions of ``goodness'' for stratified
space.
The definition below is a mild one introduced by Lurie
\cite[A.5.5]{arXiv:0911.0018}.

\begin{definition}[(Conically stratified space)]
	Let \( f \From X \to A \) be a stratified topological space. 
	One says that \( X \) is conically stratified whenever
	for each \( p \in A \) and each \( x \in X_p \), there exists
	an open neighborhood \( U_p \subset X_p \) of \( x \) and
	a stratified space \( L \) over \( P_{p <} \) such that
	\( U_p \subset X_p \) can be extended to a stratified
	space over the poset map
	\( P_{p<}^\triangleleft \IsCanonicallyIsomorphicTo
	P_{p \leq} \subset P \).
\end{definition}

We have already seen that the minimal exponential
\( \ExpMin(X) \)
is never stratified.
Even though the topological exponential is always
a stratified space over \( \PointedOmega \) when
\( X \) is separated,
it is usually impossible for the topological exponential
\( \TopExp(X) \)
to be conically stratified; conical opens would allow
sequences with unbounded cardinality to converge
\UnskipRef{Remark: converging sequence in colimit topology}
\cite[2.14]{arXiv:2102.12325}.
We shall then restrict our attention to the metric exponential
\( \MetricExp(X) \) and show that it is conically stratified
for a large class of spaces \( X \).

\begin{lemma}
	Let \( V \) be a normed vector space, then
	the function \( \Reals_+ \times V \times \MetricExp(V)
	\to \MetricExp(V) \) sending a triple
	\( (\lambda, v, S) \) to the configuration
	\[
		\lambda S + v \coloneqq \{ \lambda s + v \mid
		s \in S \}
	\]
	is continuous.
\end{lemma}

\begin{proof}
	One has
	\[
		D(\lambda S + v, \mu T + w) 
		\leq \norm{v - w} + \lambda D(S,T) + |\lambda - \mu|
		D(0, T)
	\]
	which shows that the function is continuous.
\end{proof}

\begin{proposition}
	\label{vector space cone}
	Let
	\( V \) be a rotund vector space and let us denote by
	\( \MetSphere(V) \subset \MetricRan(V) \) the subspace
	of configurations whose minimal enclosing ball has center
	\( 0 \)
	and radius \( 1 \).
	Since such a configuration must have at least two points,
	\( \MetSphere(V) \) is naturally stratified over the poset
	\( \OrdinalOmega_{2 \leq} \).
	Then, one has a canonical homeomorphism
	\[
		\MetricRan(V)
		\IsCanonicallyIsomorphicTo V \times \Cone(\MetSphere(V))
	\]
	over the isomorphism
	\( \OrdinalOmega_{1 \leq} \IsCanonicallyIsomorphicTo
	\OrdinalOmega_{2 \leq}^\triangleleft \).
\end{proposition}

\begin{proof}
	In both cases the map sends one point configurations
	\( v \in V \) to the tuple \( (v, 0) \) where \( 0 \) represents
	the tip of the cone, and sends
	multiple point configurations
	\( S \subset V \) to the tuple
	\( (\Center_S, (\Radius_S, \Radius_S^{-1}(S-\Center_S))) \).
	The inverse map simply sends
	tuples \( (v,(\lambda, S)) \) to \( \lambda S + v \).

	By the previous lemma and
	since \( S \mapsto \Center_S \) and \( S \mapsto \Radius_S \)
	are continuous, it is clear that the bijection
	restricts to a homeomorphism
	between the open subspace of non-punctual configurations
	on one side and the product of \( V \) with the interior
	of the cone on the other side.

	Finally, if \( S_n \to v \) is a converging sequence
	with limit a punctual configuration, then by continuity
	\( \Center_{S_n} \to v \) and
	\( \Radius_{S_n} \to 0 \), which means that the image of
	\( S_n \) converges to \( (v, 0) \) by definition of
	the topology of the cone.
	Conversely, if \( (v_n,(\lambda_n, S_n)) \)
	is a sequence converging
	to \( (v, 0) \), this means by definition of the topology of the
	cone that \( \lambda_n \to 0 \) and since
	\( S_n \) is bounded,
	then \( \lambda_n S_n \to 0 \) so that
	\( v_n + \lambda_n S_n \to v \)
	in \( \MetricRan(V) \).
\end{proof}

\begin{theorem}
	\label{conically}
	The exponential \( \MetricExp(X) \) is conically stratified,
	whenever
	\( X \) is a separated topological space locally homeomorphic
	to a rotund vector space.
\end{theorem}

\begin{proof} 
	Since the emptyset is a disjoint point from the rest
	of the space, it emits a conical neighborhood trivially.
	Let \( S \in \MetricRan(M) \) so that  \( |S| >0 \).
	By assumption,
	for each \( s \in S \), one
	can find an open embedding \( V_s \hookrightarrow X \) carrying
	the origin of a rotund vector space \( V_s \) to \( s \in X \).
	Moreover these can be chosen to be disjoint in \( X \).

	Since \( \MetricExp \) is an exponential which also preserves
	open embeddings,
	one can build a stratified open embedding
	\[
		\begin{tikzcd}[column sep = 25]
			\prod_{s \in S} \MetricRan(V_s) \rar[hook] \ar{d}
			&
			\prod_{s \in S} \MetricExp(V_s) \ar{r}{\IsIsomorphicTo}
			\ar{d}
			&
				\MetricExp\left(\coprod_{s \in S}V_s\right)
				\ar[hookrightarrow]{r} \ar{d}
			& \MetricExp(M) \ar{d} \\
			\prod_{s \in S} \OrdinalOmega_{1 \leq} \ar{r}{=}
			& \prod_{s \in S} \OrdinalOmega_{1 \leq} \ar{r}{+}
			& \PointedOmega \ar{r}{=}
			& \PointedOmega
		\end{tikzcd}
	\]
	whose image contains \( S \).

	Since a finite product of cones is again homeomorphic to
	a cone as a stratified space
	\UnskipRef{Product of cones} and since
	\( \MetricRan(V_s) \) is homeomorphic to
	\( V_s \times \Cone(\MetSphere(V_s)) \) for every \( s \in S \)
	\UnskipRef{vector space cone},
	one gets stratified homeomorphisms
	\[
		\begin{tikzcd}
			\prod_{s \in S} \MetricRan(V_s)
			\arrow[r, "\IsIsomorphicTo"] \arrow[d]
			& \prod_{s \in S} V_s \times \Cone(\MetSphere(V_s))
			\arrow[r, "\IsIsomorphicTo"] \arrow[d]
			& \prod_{s \in S} V_s
			\times \Cone\left(\Join_{s \in S} \MetSphere(V_s)\right)
			\arrow[d]
			\\
			\prod_{s \in S} \OrdinalOmega_{1 \leq}
			\arrow{r}{\IsIsomorphicTo}
			& \prod_{s \in S} (\OrdinalOmega_{2 \leq})^\triangleleft
			\arrow{r}{\IsIsomorphicTo}
			& \left(\Join_{s \in S}
			\OrdinalOmega_{2 \leq}\right)^\triangleleft
	\end{tikzcd}
	\]
	which concludes the proof.
\end{proof}

\begin{example}
	Since Fréchet manifolds are locally homeomorphic to Hilbert spaces
	\cite[6.1]{zbMATH03734780}
	and Hilbert spaces are rotund,
	\( \MetricExp(V) \) is conically stratified when \( V \) is
	a Fréchet manifold.
\end{example}

\begin{remark}
	Since \( \MetricExp(X) \) is conically stratified, it follows that
	each truncated version
	\( \TruncatedMetExp n (X) \) is also conically stratified.
	This truncated result was previously
	obtained by Ayala, Francis and Tanaka
	for \( X \) a smooth manifold
	\cite[3.7.5]{doi:10.4171/jems/856}.
\end{remark}

\begin{corollary}
	Let \( X \) be a metrizable space, locally homeomorphic to
	a rotund topological vector space.
	Then, the ∞-categories of \( \PointedOmega \)-constructible
	hypersheaves of spaces on
	\( \TopExp(X) \) and \( \MetricExp(X) \) are canonically
	equivalent.

	Moreover, both can be represented as the ∞-category of functors
	from the exit path ∞-category
	\( \Exit_{\PointedOmega}(\TopExp(X))
	= \Exit_{\PointedOmega}(\MetricExp(X)) \)
	to the ∞-category of spaces.
\end{corollary}

\begin{proof}
	Since we know that \( \MetricExp(X) \) is conically stratified,
	we only need to check the other axioms of the main theorem of
	\emph{Constructible hypersheaves via exit paths}
	\cite[3.13]{arXiv:2102.12325}.
	Since \( X \) is metrizable, \( \MetricExp(X) \) is also metrizable
	and thus paracompact.

	We now prove that each stratum is locally of singular shape.
	Being a local property, we can reduce to the case where \( X \)
	is homeomorphic
	to a separated locally convex topological vector space \( V \)
	\cite[A.4.16]{arXiv:0911.0018}.
	The stratum \( 0 \in \PointedOmega \) amounts to a single point
	so there is nothing to prove.

	Assume \( n \geq 1 \).
	By assumption the convex open sets form a basis of the topology of
	\( V \) which is stable under finite intersections.
	As a consequence, the opens of the form
	\( [C_s]_{s \in S} \cap \StratumExp n(V) \)
	where \( \{C_s\}_S \) is a family of \( |S| = n \)
	disjoint convex open subsets of \( V \),
	form a basis of the topology of \( \StratumExp n (V) \) which
	is stable under intersection.
	It is then enough to see that
	\( [C_s]_{s \in S} \cap \StratumExp n (V) \) has singular shape
	\cite[A.4.14]{arXiv:0911.0018}, which immediately follows from the
	fact that \( [C_s]_{s \in S} \cap \StratumExp n(V) \)
	is homeomorphic
	to \( \prod_{s \in S} C_s \) and is thus contractible.
\end{proof}

The metrizability axiom above cannot be easily removed as shown by
the following proposition.

\begin{proposition}
	There exists a paracompact topological space \( X \) for which
	neither \( \TopExp(X) \) nor \( \MetricExp(X) \) is paracompact. 
\end{proposition}

\begin{proof}
	Let \( X \) be the set
	of real numbers \( \Reals \) endowed with the lower limit
	topology: the topology whose
	basis of opens is made of the half open intervals \( [a, b) \).
	This space is paracompact but the product
	\( X^2 \) is not
	even normal
	\cite{doi:10.1090/s0002-9904-1947-08858-3}.
	Since the quotient map
	\( X^2 \to X^2_{\SymmetricGroup 2} \) is closed
	(\( \SymmetricGroup 2 \) being finite),
	\( X^2_{\SymmetricGroup 2} \) is also not normal.
	As it is a closed subset of both \( \TopExp(X) \)
	and \( \MetricExp(X) \) by the previous lemma,
	neither can be paracompact.
\end{proof}

\subsection*{Acknowledgments}

The authors would like to thank
Mathieu Anel,
David Ayala,
Damien Calaque,
Lucas Geyer,
Grégory Ginot
Jarek Kwapisz,
and
Yat-Hin Suen.

\bibliography{ms}

\end{document}